\documentclass[11pt]{article}
\usepackage{amsmath}
\usepackage{amssymb}
\usepackage{amsthm}
\usepackage{mathrsfs}
\usepackage{latexsym}
\usepackage{amsthm}
\usepackage{graphicx}
\usepackage{setspace}
\usepackage{xy}
\input xy
\xyoption{all}
\theoremstyle{plain}
\newtheorem{theorem}{Theorem}[section]
\newtheorem{corollary}[theorem]{Corollary}
\newtheorem{proposition}[theorem]{Proposition}
\newtheorem{lemma}[theorem]{Lemma}

\theoremstyle{remark}
\newtheorem{remark}[theorem]{Remark}
\theoremstyle{definition}
\newtheorem{definition}[theorem]{Definition}
\newtheorem{construction}[theorem]{Construction}

\newtheorem{example}[theorem]{Example}
\newtheorem{proposition-definition}[theorem]{Proposition-Definition}

\newtheorem{notation}[theorem]{Notation}

\newcommand{\C}{\mathbb{C}}

\newcommand{\G}{\mathbb{G}}
\newcommand{\pp}{\mathbb{P}}

\newcommand{\Q}{\mathbb{Q}}
\newcommand{\vv}{\mathrm{virt}}

\title{Stable maps and stable quotients}
\author{Cristina Manolache}
\date{}

\begin{document}
\maketitle
\begin{abstract}
We analyze the relationship between two compactifications of the moduli space of maps from curves to a Grassmannian: the Kontsevich moduli space of stable maps and the Marian--Oprea--Pandharipande moduli space of stable quotients. We construct a moduli space which dominates both the moduli space of stable maps to a Grassmannian and the moduli space of stable quotients, and equip our moduli space  with a virtual fundamental class. We relate the virtual fundamental classes of all three moduli spaces using the virtual push-forward formula. This gives a new proof of a theorem of Marian--Oprea--Pandharipande: that enumerative invariants defined as intersection numbers in the stable quotient  moduli space coincide with Gromov--Witten invariants.
\end{abstract}
\tableofcontents
\section{Introduction}The  Kontsevich moduli space of stable maps to Grassmannians and the moduli space of stable quotients of Marian--Oprea--Pandharipande are two compactifications of spaces of curves on Grassmannians. These moduli spaces come equipped with virtual classes in the sense of \cite{bf}, \cite{lg}. The purpose of this paper is to understand the relation between the two virtual fundamental classes and thus provide a new proof of a theorem in \cite{mop}: that enumerative invariants defined as virtual intersection numbers in the two moduli spaces coincide. We do this by constructing a new moduli space of \emph{map-quotients} which dominates both the moduli space of stable maps and the moduli space of stable quotients. We endow this space with a virtual class and determine its relation to the virtual classes of the moduli of stable maps and stable quotients using the virtual push-forward theorem \cite{eu2}.
\\

In the following we briefly review the main definitions and we outline the main constructions. 
\subsection{Stable maps and stable quotients} 
\paragraph{Stable maps to Grassmannians.} Let $\G(k,r)$ be the Grassmannian of $k$-planes in the $r$-dimensional affine space. Let $(C, p_1,...,p_n)$ be a nodal  curve of genus $g$ with $m$ distinct markings which are different from the nodes. By the universal property of Grassmannians giving a degree $d$ map from $C$ to $\G(k,r)$ is equivalent to giving an exact sequence
\begin{equation*}
0\to S\to \mathcal{O}^{\oplus r}\to Q\to 0
\end{equation*}
where $S$ is a rank $k$ vector bundle of degree $d$ and $Q$ is a vector bundle. A map is called stable if its degree is positive on each unstable contracted component. It has been shown in \cite{toda} that this is equivalent to $$\omega_{\hat{C}}(p_1+...+p_n)\otimes (\wedge^k\hat{S}^{\vee})^{\epsilon}$$ being ample on $C$ for $\epsilon>2$. The moduli space of degree $d$ stable genus $g$ maps with $n$ marked points to $\G(k,r)$ wil be denoted by $\overline{M}_{g,n}(\G(k,r),d)$. 
\paragraph{Stable quotients.} Let $(\hat{C},p_1,...,p_n)$ be a nodal curve of genus $g$ with $n$ distinct markings which are different from the nodes. A quotient on $\hat{C}$ $$0\to \hat{S}\to\mathcal{O}_{\hat{C}}^{\oplus r}\stackrel{q}{\rightarrow} \hat{Q}\to 0$$ is called \textit{quasi-stable} if the torsion sheaf $\tau(\hat{Q})$ is not supported on nodes or markings. Let $k$ be the rank of $S$. A quotient $(\hat{C},p_1,...,p_n,q)$ is called stable if $$\omega_{\hat{C}}(p_1+...+p_n)\otimes (\wedge^k\hat{S}^{\vee})^{\epsilon}$$ is ample on $\hat{C}$ for every strictly positive $\epsilon \in\Q$. The moduli space of degree $d$ stable genus $g$ quotients with $n$ marked points will be denoted by $\overline{Q}_{g,n}(\G(k,r),d)$. 

The space $\overline{Q}_{g,n}(\G(k,r),d)$ is another compactification of the space of genus $g$ curves with $n$ marks in the Grassmannian $\G(k,r)$. 

\paragraph{Morphisms between moduli spaces of stable maps and moduli spaces of stable quotients.} Marian--Oprea--Pandharipande showed in \cite{mop} that if $k=1$ then there exists a morphism of stacks $$c:\overline{M}_{g,n}(\G(k,r),d)\to \overline{Q}_{g,n}(\G(k,r),d).$$ On points $c$ is obtained in the following way. Let $C^0_i$ be the rational tails without marked points of a stable maps $(C,p_1,...,p_n,f:C\to\G(k,r))$.
Let us suppose that the degree of $f$ restricted to $C^0_i$ is $d_i$. Let $\hat{C}$ be the closure of $C\backslash C^0_i$ in $C$ and let $p:C\to\hat{C}$ be the morphism contracting $C^0_i$. Let $\hat{S}=p_*S(-d_iC^0_i)$. Then to the stable map $0\to S\to \mathcal{O}^n\to Q\to 0$ $c$ associates a stable quotient
\begin{equation*}
0\to \hat{S}\to\mathcal{O}_{\hat{C}}^{\oplus r}\stackrel{q}{\rightarrow} \hat{Q}\to 0.
\end{equation*}
We show that for $k>1$ there is no such morphism: see Example \ref{popa}.

\subsection{Stable map-quotients} Below we define a proper DM-stack $\overline{MQ}_{g,n}(\G(k,r),d)$ with the following properties
\begin{enumerate}
\item $\overline{MQ}_{g,n}(\G(k,r),d)$ admits natural morphisms 
\begin{equation*}
\xymatrix{&\overline{MQ}_{g,n}(\G(k,r),d)\ar[ld]_{c_1}\ar[rd]^{c_2}\\
\overline{M}_{g,n}(\G(k,r),d)&&\overline{Q}_{g,n}(\G(k,r),d)}
\end{equation*}
\item $\overline{MQ}_{g,n}(\G(k,r),d)$ admits a dual relative obstruction theory $E^{\bullet}_{\overline{MQ}}$ (relative to some pure dimensional stack) which comes equipped with morphisms
\begin{align}\label{inspir}E^{\bullet}_{\overline{MQ}}\to c_1^* E^{\bullet}_{\overline{M}}\\E^{\bullet}_{\overline{MQ}}\to c_2^* E^{\bullet}_{\overline{Q}}.
\end{align}
\end{enumerate}
Having constructed such a stack allows us to relate virtual classes by means of the virtual push-forward property \cite{eu2}. In the following we outline the structure of this paper. 

In Section \ref{mbundles} we review moduli spaces $\mathfrak{Bun}_{g,n}(k,d)$ of rank $k$, degree $d$ vector bundles on genus $g$ nodal curves with $n$ marked points and introduce an auxiliary moduli space $\widetilde{\mathfrak{Bun}}_{g,n}(k,d)$ which comes equipped with a birational morphism $\pi_1:\widetilde{\mathfrak{Bun}}_{g,n}(k,d)\to \mathfrak{Bun}_{g,n}(k,d)$. In particular $\widetilde{\mathfrak{Bun}}_{g,n}(k,d)$ has pure dimension. 

In Section \ref{map-quot} we construct a proper Deligne-Mumford stack $\overline{MQ}_{g,n}(\G(k,r),d)$ which is a substack of $\overline{M}_{g,n}(\G(k,r),d)\times_{\mathfrak{M}^{rtf}_{g,n}}\overline{Q}_{g,n}(\G(k,r),d)$. This stack fits into a commutative diagram
\begin{equation*}
\xymatrix{
&{\overline{MQ}_{g,n}(\G(k,r),d)}\ar[dd]\ar[ld]_{c_1}\ar[rd]^{c_2}\\
\overline{M}_{g,n}(\G(k,r),d)\ar[dd]&&\overline{Q}_{g,n}(\G(k,r),d)\ar[d]\\
&\widetilde{\mathfrak{Bun}}_{g,n}(k,d)\ar[ld]_{\pi_1}\ar[r]^{\pi_2}&\mathfrak{Bun}_{g,n}(k,d)\\
\mathfrak{Bun}_{g,n}(k,d).}
\end{equation*}
Moreover, the rectangle on the left is cartesian and thus it gives rise to a perfect obstruction theory of $\overline{MQ}_{g,n}(\G(k,r),d)$ relative to $\widetilde{\mathfrak{Bun}}_{g,n}(k,d)$. This gives rise to a virtual class on $\overline{MQ}_{g,n}(\G(k,r),d)$. 

In Section \ref{compvf} we use the virtual push-forward theorem to show that $c_1$ and $c_2$ satisfy the virtual push forward property: see Theorem \ref{final}, which gives a new proof of the Marian--Oprea--Pandharipande theorem mentioned above.

One of the main technical difficulties is to make the constructions of $\widetilde{\mathfrak{Bun}}_{g,n}(k,d)$ and $\overline{MQ}_{g,n}(\G(k,r),d)$ functorial. We will do slightly less, namely we will construct functorial spaces $\mathfrak{P}$ and $\bar{P}$ which contain  $\widetilde{\mathfrak{Bun}}_{g,n}(k,d)$ and $\overline{MQ}_{g,n}(\G(k,r),d)$ respectively. 
\paragraph{Relation to other works} In the past years many birational models of moduli spaces of stable maps have been constructed for particular targets. These include: moduli spaces of weighted stable maps introduced by Bayer and Manin \cite{am}, the moduli spaces defined by Musta\c{t}\u{a}--Musta\c{t}\u{a} \cite{mm}, moduli spaces of stable quotients of Marian, Oprea and Pandharipande \cite{mop} with a more general version introduced by Toda \cite{toda}, moduli of stable toric quasi-maps \cite{cf1} and finally, moduli spaces of stable quasi-maps to GIT quotients \cite{cf2} introduced by Ciocan-Fontanine, Kim and Maulik which generalize \cite{mop}, \cite{toda} and \cite{cf1}. These spaces are particularly interesting because they represent some natural functors and because they lead to invariants which are closely related to Gromov-Witten invariants. The stable quasi-map invariants (or variants of them) are also easier to compute in some cases (see \cite{gi}).
\\It is therefore interesting to compare quasi-map invariants to Gromov-Witten invariants. This has already been done for stable quotients in \cite{mop} and \cite{toda} by localization. Our approach is completely different and the main hope is that it will shed light on similar questions. More precisely, we first construct (a rather unnatural) auxiliary moduli space with a virtual class and then relate this virtual class to the virtual classes of the original spaces. We emphasize that this can be done with very little information on the auxiliary moduli space (e.g we do not know whether it has an absolute perfect obstruction theory).
\\As regarding the birational geometry of moduli spaces of stable maps to GIT quotients little is known: divisors on moduli spaces of maps stable maps were studied mostly in genus zero and for homogeneous spaces (see e.g. \cite{chs1}, \cite{chs2}, \cite{chs3}). The special feature in our case is that there is a morphism from an open dense set of the moduli space of maps to Grassmannians to the moduli space of stable quotients. This morphism does not extend in general. The proofs of these facts are essentially the same as the ones in \cite{mihnea}.
\\As a final remark, one could hope that the techniques in this paper easily extend to the case of quasi-maps to GIT quotients. This is unfortunately not true as in general there is no map from an open dense set of the moduli space of stable maps to the moduli space of qusi-maps. This fact introduces extra challenges at level of virtual classes comparisons and it will be treated elsewhere. 
\paragraph{Acknowledgements} I would like to thank G Farkas, A Ortega, T Coates, A Corti, E Macr\`i for useful discussions. I am particularly grateful to I Ciocan-Fontanine for pointing out several delicate issues and for very inspiring discussions. Many thanks to M Popa for explaining to me some aspects of his paper \cite{mihnea} and to B Fantechi to whom I owe most of Sections \ref{barbara} and \ref{barbara2}.
\\I was supported by SFB-$647$ and by a Marie Curie Intra-European Fellowship: FP7-PEOPLE-2011-IEF.
\section{Moduli of Bundles}\label{mbundles}
\subsection{Moduli of bundles over nodal curves}\label{barbara}
We review a few results concerning the existence and properties of stacks of vector bundles over prestable curves.
\paragraph{Moduli of prestable curves.} Let us first fix notations. Let $g\geq 0$ and $n\geq 0$ be integers. We denote by $\mathfrak{M}_{g,n}$ the Artin stack of prestable curves of genus $g$ with $n$ marked points. As in the case of stable curves, $\mathfrak{M}_{g,n+1}$ is the universal curve.
\begin{definition} Let $C$ be a nodal curve.  A connected rational component (not necessarily irreducible) $C^0$ with no marked points such that $C^0$ intersects the rest of the curve in exactly one point is called a rational tail.
\end{definition}
\begin{definition}In the following we denote by $\mathfrak{M}_{g,n}^{rt}$ the divisor of $\mathfrak{M}_{g,n}$ whose points are curves which have rational tails and by $\mathfrak{M}_{g,n}^{rtf}$ be the open substack of $\mathfrak{M}_{g,n}$, whose points are rational tail free curves.
\end{definition}
\begin{proposition}\label{contr} Let $\mathfrak{S}\subset \mathfrak{M}_{g,n}$ be a substack of finite type of $\mathfrak{M}_{g,n}$. Then, there exists a morphism of stacks $p:\mathfrak{S}\to\mathfrak{M}_{g,n}^{rtf}$ which contracts rational tails.
\end{proposition}
\begin{proof} Let $\mathfrak{S}^{rtf}$ be the open substack of $\mathfrak{S}$, whose points are rational tail free curves. Let $\mathfrak{S}^{\prime}_{g,n}$ be the set of pairs $(C,\hat{C})\in\mathfrak{S}\times\mathfrak{S}^{rtf}$ such that there exists $p:C\to\hat{C}$ which contracts rational tails with no marked points and is the identity on the complement of such curves. In order to show that $\mathfrak{S}^{\prime}$ is a substack of $\mathfrak{S}\times\mathfrak{S}^{rtf}$ we show that we are under the hypothesis of Example 4.19 in \cite{fga}. It is clear that any cartesian diagram arrow whose target is in $\mathfrak{S}^{\prime}$ is also in $\mathfrak{S}^{\prime}$. Let now $B_i$ be a covering of $B$ and suppose that $\mathfrak{S}\times\mathfrak{S}^{rtf}(B_i)\in \mathfrak{S}_{g,n}(B_i)$. Then we need to show that  $\mathfrak{S}\times\mathfrak{S}^{rtf}(B_i)\in \mathfrak{S}^{\prime}(B)$. As $p_i$ is the identity away from rational tails we can glue along $p_i$ to get a global $p:C\to\hat{C}$.
\\The algebraic $\mathfrak{S}^{\prime}$ has a projection $p_2:\mathfrak{S}^{\prime}\to\mathfrak{S}^{rtf}$. We next prove that $$\mathfrak{S}^{\prime}\simeq\mathfrak{S}.$$ We claim that $p_1:\mathfrak{S}^{\prime}\to\mathfrak{S}$ is separated. As $p_1$ is also one to one we have that $p_1$ is an isomorphism. Let us sketch the proof of the claim. It is enough to show that given  a family of curves $\mathcal{C}\to\Delta$ and a projection $p:\mathcal{C}\to\hat{\mathcal{C}}$ over $\Delta^*$, $p$ extends uniquely over $\Delta$. It is clear that any component of $\mathcal{C}$ whose fibers are rational tails can be contracted so we may assume that $\mathcal{C}$ is smooth. As rational tails in $\mathcal{C}$ are chains of $\pp^1$'s with negative self intersection there exists a unique contraction map $p:\mathcal{C}\to\hat{\mathcal{C}}$.   
\end{proof}
\paragraph{Moduli of bundles over prestable curves.} Let $\mathfrak{Bun}_{g,n}(k,d)(B)$ be the category whose objects are pairs $(\mathcal{C},\mathcal{S})$, where 
 \begin{itemize}
 \item $\mathcal{C}\to B$ is a family of prestable curves of genus $g$  with $n$ sections
 \item $\mathcal{S}$ is a vector bundle of rank $k$ and degree $d$.
 \end{itemize}
Isomorphisms:  An isomorphism
\begin{equation*}
\phi:(\mathcal{C},\mathcal{S})\to(\mathcal{C}^{\prime},\mathcal{S}^{\prime})
\end{equation*}
is an automorphism of curves
\begin{equation*}
\phi:\mathcal{C}\to\mathcal{C}^{\prime}
\end{equation*}
together with isomorphisms $\theta: \mathcal{S}^{\vee}\to \phi^*{\mathcal{S}^{\prime}}^{\vee}$ such that $\phi(p_i)=p_i^{\prime}$ , $\forall i$.

By \cite{ml} we have that $Coh_{\mathfrak{M}_{g,n+1}/\mathfrak{M}_{g,n}}$ is an Artin stack. It can be easily seen that $\mathfrak{Bun}_{g,n}(k,d)$ is  a substack of $Coh_{\mathfrak{M}_{g,n+1}/\mathfrak{M}_{g,n}}$ (see \cite{cf2} for more details and generalizations). Let $\mathcal{S}$ denote the universal bundle on the universal curve on $\mathfrak{Bun}_{g,n}(k,d)$. 
We will also consider moduli spaces of vector bundles on curves with stability conditions as follows.
\begin{construction} Let $\mathfrak{Bun}^{\epsilon}_{g,n}(k,d)$ be the substack of $\mathfrak{Bun}_{g,n}(k,d)(B)$ such that the line bundle\begin{equation}
(\wedge^k\mathcal{S}^{\vee})^{\otimes \epsilon}\otimes \omega(\sum p_i)
\end{equation}
is ample. As ampleness is an open condition $\mathfrak{Bun}^{\epsilon}_{g,n}(k,d)$ is an open substack of $\mathfrak{Bun}_{g,n}(k,d)$.
\end{construction}

\begin{remark}\label{smooth} Let $$\phi:\mathfrak{Bun}_{g,n}(k,d)\to \mathfrak{M}_{g,n}$$ be the morphism which forgets the bundle. The morphism $\phi$ is smooth as the relative obstruction in a point $(C, S)$ is $$Ext_C^2(S, S)=0.$$ This shows that $\mathfrak{Bun}_{g,n}(k,d)$ is smooth of pure dimension $Ext^1(S, S)-Ext^0(S, S)+3g-3+n=k(g-1)-deg(S\otimes S^{\vee})+3g-3+n=k^2(g-1)+3g-3+n$.  
\end{remark}

\begin{construction}Consider the stack $\mathfrak{Bun}_{g,n}(k,d)^{rt}$ defined by the following cartesian diagram
\begin{equation*}
\xymatrix{\mathfrak{Bun}_{g,n}(k,d)^{rt}\ar[r]\ar[d]&\mathfrak{Bun}_{g,n}(k,d)\ar[d]^{\phi}\\
\mathfrak{M}_{g,n}^{rt}\ar[r]^i&\mathfrak{M}_{g,n}.}
\end{equation*}
Let $\mathfrak{Bun}_{g,n}(k,d)^{rtf}$ be the complement of $\mathfrak{Bun}_{g,n}(k,d)^{rt}$ in $\mathfrak{Bun}_{g,n}(k,d)$.
\end{construction}
\begin{remark} We have that $\mathfrak{M}_{g,n}^{rt}$ has codimension 1 in $\mathfrak{M}_{g,n}$. Remark \ref{smooth} implies that $\mathfrak{Bun}_{g,n}(k,d)^{rt}$ has codimension 1 in $\mathfrak{Bun}_{g,n}(k,d)$.
\end{remark}
\begin{lemma}\label{contraction} Let $\mathcal{C}$ be the universal curve over $\mathfrak{Bun}^{\epsilon}_{g,n}(k,d)$. Then there exists a rational tail free curve $\hat{\mathcal{C}}$ and a projection $p:\mathcal{C}\to\hat{\mathcal{C}}$ over $\mathfrak{Bun}^{\epsilon}_{g,n}(k,d)$.
\end{lemma}
\begin{proof} Let $\pi:\mathcal{C}\to B, \mathcal{S}$ be the tautological bundle on the tautological curve of $\mathfrak{Bun}^{\epsilon}_{g,n}(k,d)$. Without loss of generality we may assume that the divisor consisting of curves with rational tails is irreducible, otherwise we repeat the construction for each component. Let $a$ be the degree of $\wedge^k\mathcal{S}^{\vee}$ restricted to the locus consisting of rational tails. As $(\wedge^k\mathcal{S}^{\vee})^{\otimes \epsilon}\otimes \omega(\sum p_i)$ is ample we have that $$\mathcal{L}=(\wedge^k\mathcal{S}^{\vee})^{\otimes \epsilon}\otimes \omega(\sum p_i)^{\epsilon a}$$ is trivial on the locus consisting of curves with one rational tail. As $\mathcal{C}$ is relatively normal it follows that $\mathcal{L}$ is trivial on all rational tails and $\pi$ relatively ample on the complement of this locus. This shows that $\mathcal{L}^m$ is base point free for a sufficiently large $m$. Let $$\hat{\mathcal{C}}=\mathrm{Proj}\oplus_l \mathcal{L}^{ml}.$$ As $\mathcal{L}^m$ is $\pi$-relatively base point free it determines a morphism $p:\mathcal{C}\to \hat{\mathcal{C}}$. We have that $\hat{\mathcal{C}}\to B$ is a family of genus $g$ curves and as $\mathfrak{Bun}^{\epsilon}_{g,n}(k,d)$ is reduced, we obtain that $\hat{\mathcal{C}}$ is flat over $\mathfrak{Bun}^{\epsilon}_{g,n}(k,d)$.
\end{proof}

\subsection{The most balanced locus}\label{barbara2}
\begin{construction}\label{abal}Let $\mathfrak{Bun}_{0,n}(k,d)^{bal}$ be the substack of $\mathfrak{Bun}_{0,n}(k,d)$ such that on every component of a rational curve $C$ the bundle $S$ is the most balanced one. More precisely, if $C_i$ is a rational component of $C$ such that $S_i:=S_{|C_i}$ has degree $d_i$, then $$S_i=\mathcal{O}(a_i)^{\oplus k_1}\oplus\mathcal{O}(a_i-1)^{\oplus k-k_1},$$ where $a_i$, $k_1$ are the unique integers such that  $k_1a_i+(k-k_1)(a_i-1)=d_i$. We call $(C_i,S_i)$ as above $a_i$-balanced.
\end{construction}
\begin{lemma}$\mathfrak{Bun}_{0,n}(k,d)^{bal}$ is an open substack of $\mathfrak{Bun}_{0,n}(k,d)$.
\end{lemma}
\begin{proof} Let $\pi:\mathfrak{C}\to\mathfrak{Bun}_{0,n}(k,d)$ be the universal curve of $\mathfrak{Bun}_{0,n}(k,d)$. The non-balanced locus of $\mathfrak{Bun}_{0,n}(k,d)$ is the support of the sheaf $\pi_*Ext^1(\mathcal{S}, \mathcal{S})$.
\end{proof}
\begin{definition}Let $\mathfrak{Bun}_{g,n}(k,d)^{bal}$ be the substack of $\mathfrak{Bun}_{g,n}(k,d)$ whose points are bundles on curves, which are balanced on the rational tails as explained above.
\end{definition}
\begin{proposition} \label{bunbil}$({\mathfrak{Bun}_{g,n}(k,d)^{\epsilon}})^{bal}$ is an open substack of $\mathfrak{Bun}_{g,n}(k,d)$ and the complement has codimension at least 2.
\end{proposition}
\begin{proof}
This follows just as before. Let $p:\mathcal{C}\to\hat{\mathcal{C}}$ the morphism from Lemma \ref{contraction} and let $\mathfrak{Z}$ be the support of the sheaf $p_*Ext^1(\mathcal{S}, \mathcal{S})$. Then ${\mathfrak{Bun}_{g,n}(k,d)^{\epsilon}}^{bal}$ is the complement of $\mathfrak{Z}$ in $\mathfrak{Bun}_{g,n}(k,d)$ and $\mathfrak{Z}$ has codimension at least 2.
\end{proof}

We recall here a cohomology and base change Lemma from \cite{mihnea} (Lemma 7.1 in \cite{mihnea}).
\begin{lemma}\label{cbpopa}
If $\mathcal{S}$ is a vector bundle on a family of curves $\mathcal{C}\to B$ such that $R^1p_*\mathcal{S}=0$, then $p_*\mathcal{S}$ is a flat family of coherent sheaves over $B$ and formation of $p_*\mathcal{S}$ commutes with arbitrary base change $B^{\prime}\to B$. 

\end{lemma}
\begin{proposition} There exists a morphism $$r: \mathfrak{Bun}^{\epsilon}_{g,n}(k,d)^{bal}\to \mathfrak{Bun}^{\epsilon}_{g,n}(k,d)^{rtf}.$$
\end{proposition}
\begin{proof} The proof follows essentially from \cite{mihnea}. Let $\mathcal{C}$ be a flat family of curves over a base scheme $B$. Let $D_i$ be the intersection of the divisor $\mathfrak{Bun}^{\epsilon}_{g,n}(k,d)^{rt}$ with $\mathcal{C}$. Suppose that on each rational tail $\mathcal{S}$ splits on $D_i$ as $$\oplus^{k_1}\mathcal{O}(a_i)\oplus^{k-k_1}\mathcal{O}(a_i+1).$$ Then $\mathcal{R}^1p_*\mathcal{S}((a_i+1)D_i)=0$ and from Lemma 7.1. in \cite{mihnea} we obtain that $p_*\mathcal{S}((a_i+1)D_i)$ is a vector bundle.
\end{proof}

\subsection{Construction of $\widetilde{\mathfrak{Bun}}_{g,n}(k,r,d)$}
The goal of this section is to find a stack $\widetilde{\mathfrak{Bun}}_{g,n}(k,r,d)$ which surjects to $\mathfrak{Bun}_{g,n}(k,d)$ and a morphism $\widetilde{\mathfrak{Bun}}_{g,n}(k,r,d)\to \mathfrak{Bun}_{g,n}(k,d)^{rtf}$ which extends the morphism $r$.
\begin{notation}\label{notation} \label{abuse}Let $\mathcal{C}$ be a family of nodal curves over a scheme $B$. We denote by $\mathcal{U}$ the complement of the locus of rational tails in $\mathcal{C}$. By Proposition \ref{contr} there exists a family of rational tail free curves $\hat{\mathcal{C}}$ and a morphism $p:\mathcal{C}\to\hat{\mathcal{C}}$. Note that $\mathcal{U}$ can be identified canonically with the complement of the image of $p$ in $\hat{\mathcal{C}}$. By abuse of notation we say that $\mathcal{U}$ is also an open subset of $\hat{\mathcal{C}}$.  
\end{notation}


\begin{construction} \label{funcp}
Let $\mathfrak{P}(B)$ be the category whose objects are
\begin{equation*}
(\mathcal{C}\to B,p:\mathcal{C}\to \hat{\mathcal{C}}\to B,p_1,...,p_n,\mathcal{S}, \hat{\mathcal{S}}, p_*(\mathcal{S}^{\vee})\stackrel{\rho}{\rightarrow} \hat{\mathcal{S}}^{\vee})
\end{equation*}
with $\mathcal{C}$,  $\hat{\mathcal{C}}$ flat families of curves of genus $g$, $\mathcal{S}$ is a vector bundle of rank $k$ and degree $d$ on $\mathcal{C}$ and $\hat{\mathcal{S}}$ is a vector bundle of rank $k$ and degree $d$ on $\hat{\mathcal{C}}$, such that the following conditions hold
\begin{enumerate}
\item $p:\mathcal{C}\to \hat{\mathcal{C}}$ is the morphism of $B$-schemes from lemma \ref{contraction}

\item the restriction of $\rho$ to $\mathcal{U}$  is an isomorphism
\item $\omega(\sum p_i)\otimes \wedge^k S^{\epsilon}$ is ample for $\epsilon>2$
\item for each morphism $$p_*(S^{\vee})\stackrel{\rho}{\rightarrow} \hat{S}^{\vee}\to\tau\to0$$ and $P\in \hat{C}$ such that the fiber of $p$ over $P$ is one dimensional we have that $$length(\tau_P)>0.$$
\item the number of components of $\mathcal{C}$ is bounded by some $N$.
\end{enumerate}
{\bf Isomorphisms:} An isomorphism
\begin{equation*}
\psi:(p:\mathcal{C}\to\hat{\mathcal{C}},\mathcal{S},\rho:p_*\mathcal{S}^{\vee}\to\hat{\mathcal{S}}^{\vee})\to(p^{\prime}:\mathcal{C}^{\prime}\to\hat{\mathcal{C}}^{\prime},\mathcal{S}^{\prime}, \rho^{\prime}:p_*\hat{\mathcal{S}}^{\prime}\to{\mathcal{S}^{\vee}}^{\prime})
\end{equation*}
is an automorphism of curves
\begin{align*}
\phi:\mathcal{C}&\to\mathcal{C}^{\prime}\\
\hat{\phi}:\hat{\mathcal{C}}&\to\hat{\mathcal{C}}^{\prime}
\end{align*}
together with isomorphims  and $\theta: \mathcal{S}^{\vee}\to \phi^*{\mathcal{S}^{\prime}}^{\vee}$
 and $\hat{\theta}: \hat{\mathcal{S}}^{\vee}\to \phi^*\hat{\mathcal{S}^{\prime}}^{\vee}$
such that
\begin{enumerate}
\item $\phi(p_i)=p_i^{\prime}$ , $\forall i$
\item $p^{\prime}\circ\phi=\hat{\phi}\circ p$
\item the following diagram commutes
\begin{equation*}
\xymatrix{p_*\mathcal{S}^{\vee}\ar[r]^{\rho}\ar[d]_{\hat{\theta}_p}&\hat{\mathcal{S}}^{\vee}\ar[d]^{\hat{\theta}}\\
\hat{\phi}^*p_*{\mathcal{S}^{\prime}}^{\vee}\ar[r]^{\hat{\phi}^*\rho^{\prime}}&\hat{\phi}^*\hat{\mathcal{S}^{\prime}}^{\vee}}
\end{equation*}
where $\hat{\theta}_p:p_*\mathcal{S}^{\vee}\to p_*\phi^*{\mathcal{S}^{\prime}}^{\vee}$ is the isomorphism of sheaves induced by $\theta$ followed by the isomorphism $p_*\phi^*{\mathcal{S}^{\prime}}^{\vee}\simeq\hat{\phi}^*p_*{\mathcal{S}^{\prime}}^{\vee}$ from Lemma \ref{cbpopa}.
\end{enumerate}
\end{construction}
\begin{remark} Although $\mathfrak{P}$ and subsequently $\widetilde{\mathfrak{Bun}}_{g,n}(k,r,d)$  depends on $N$ we will omit $N$ from the notation. In the next section $N$ will be some large enough number (see Remark \ref{bounded} for more details).
\end{remark}
Let $\varphi:B^{\prime}\to B$ and let us consider the following cartesian diagram
\begin{equation*}
\xymatrix{\mathcal{C}^{\prime}\ar[r]^{\varphi}\ar[d]_{p^{\prime}}&\mathcal{C}\ar[d]^p\\
\hat{\mathcal{C}^{\prime}}\ar[r]^{\hat{\varphi}}&\hat{\mathcal{C}}
}
\end{equation*}
\begin{lemma} \label{geninj}Let $\mathcal{S}$ and $\mathcal{S}^{\prime}$ be locally free sheaves on a flat family of curves $\mathcal{C}\to B$ and $f:\mathcal{S}\to\mathcal{S}^{\prime}$ an injective morphism of sheaves. If $f$ is injective at the general point of every fiber of $\mathcal{C}\to B$, then the quotient
\begin{equation}\label{geninje}
0\to \mathcal{S}\stackrel{f}{\rightarrow}\mathcal{S}^{\prime}\to \mathcal{Q}\to 0
\end{equation}
is flat over $B$.
\end{lemma}
\begin{proof} This follows from \cite{mihnea}. Let $\mathcal{I}$ be an ideal sheaf of a subscheme of $B$. Tensoring (\ref{geninje}) with $\mathcal{Q}$ we obtain a sequence
\begin{equation*}
0\to Tor_1(\mathcal{O}_B/\mathcal{I},\mathcal{Q})\to\mathcal{S}\otimes\mathcal{O}_B/\mathcal{I}\stackrel{}{\rightarrow}\mathcal{S}^{\prime}\otimes\mathcal{O}_B/\mathcal{I}\to \mathcal{Q}\otimes\mathcal{O}_B/\mathcal{I}\to 0.
\end{equation*}
From the hypothesis we have that the map $\mathcal{S}\otimes\mathcal{O}_B/\mathcal{I}\stackrel{}{\rightarrow}\mathcal{S}^{\prime}\otimes\mathcal{O}_B/\mathcal{I}$ is injective and therefore $Tor_1(\mathcal{O}_B/\mathcal{I},\mathcal{Q})=0$.
\end{proof}
\begin{lemma} Let $\varphi:B^{\prime}\to B$ and let $(\mathcal{C}, \mathcal{S},\hat{\mathcal{S}},p_*\mathcal{S}^{\vee}\rightarrow \hat{\mathcal{S}}^{\vee})$ be an object of $\mathfrak{P}(B)$. Let $\mathcal{C}^{\prime}:=\mathcal{C}\times_BB^{\prime}$, $\hat{\mathcal{C}}^{\prime}:=\hat{\mathcal{C}}\times_{B^{\prime}}B$, $U^{\prime}=U\times_BB^{\prime}$,$\mathcal{S}^{\prime}:=\varphi^*\mathcal{S}$ and $\hat{\mathcal{S}}^{\prime}:=\hat{\varphi}^*\hat{\mathcal{S}}$. Then the natural morphism $\rho^{\prime}:p^{\prime}_*{\mathcal{S}^{\prime}}^{\vee}\rightarrow \hat{\mathcal{S}^{\prime}}^{\vee}$ is an isomorphism on $U^{\prime}$.
\end{lemma}
\begin{proof} By condition (2) we have that $$\hat{\varphi}^*\rho:\hat{\varphi}^*p_*(\mathcal{S}^{\vee})\rightarrow \hat{\varphi}^*\hat{\mathcal{S}}^{\vee}$$ is generically injective in every fiber of $\mathcal{C}\to B$. By Lemma \ref{geninj} we have that the restriction of $\varphi^*\rho$ is injective on $U$.  As $\rho$ is surjective on $U$, we have that $\hat{\varphi}^*\rho$ is surjective on $U$. This shows that $\varphi^*\rho$ is an isomorphism on $U$. As $H^1(C^0,\mathcal{S}^{\vee})=0$ for any fiber $C$, we have by Lemma \ref{cbpopa} that cohomology commutes with base change. This gives $$\hat{\varphi}^*p_*(\mathcal{S}^{\vee})\simeq p^{\prime}_*\varphi(\mathcal{S}^{\vee}).$$ Combining the two relation we obtain the conclusion.
\end{proof}

\begin{proposition} The functor $\mathfrak{P}$ is an Artin stack.
\end{proposition}
\begin{proof} 
Let $\mathfrak{S}\subset \mathfrak{M}_{g,n}$ be the subset of finite type of curves with at most $N$ components and let $\pi:\mathcal{C}\to\mathfrak{S}$ be the universal curve. Consider $\mathfrak{X}_m=Quot_P(\mathcal{O}(-m)^{P(m)})$ the $\pi$-relative $Quot$ scheme and let $\mathcal{F}$ be the universal quotient. Similarly, let $\hat{\mathfrak{S}}\subset\mathfrak{M}^{rtf}_{g,n}$ the image of $\mathfrak{S}$ via the morphism which contracts rational tails and $\pi:\hat{\mathcal{C}}\to\hat{\mathfrak{S}}$ be the universal curve. Consider $\hat{\mathfrak{X}}_m=Quot_m(\mathcal{O}(-m)^{P(m)})$ the $\pi$-relative $Quot$ scheme. By \cite{ml} we have that  $Hom_{\hat{\mathfrak{X}}_l\times_{\hat{\mathfrak{S}}} \hat{\mathfrak{X}}_m}(\mathcal{O}(-l)^{P(l)},\mathcal{O}(-m)^{P(m)})$ is an Artin stack. Let $\mathfrak{H}$ be the locally closed locus in $$Hom_{\hat{\mathfrak{X}}_l\times_{\hat{\mathfrak{S}}} \hat{\mathfrak{X}}_m}(\mathcal{O}(-l)^{P(l)},\mathcal{O}(-m)^{P(m)})$$ defined by the the following
\begin{enumerate}
\item the universal quotient $\mathcal{F}^{\prime}$ on the second factor is locally free
\item $\mathcal{F}(l)$ and $\mathcal{F}^{\prime}(m)$ are generated by global sections and higher cohomology of $\mathcal{F}(l)$ and $\mathcal{F}^{\prime}(m)$ vanishes
\item the morphism $\mathcal{O}(-l)^{P(l)}\to\mathcal{O}(-m)^{P(m)}$ induces a morphism $$\rho:\mathcal{F}\to\mathcal{F}^{\prime}$$

\item $\omega(\sum p_i)\otimes \wedge^k {\mathcal{F}}^{\epsilon}$ is ample for $\epsilon>2$
\item for each morphism $$\mathcal{F}\stackrel{\rho}{\rightarrow} \mathcal{F}^{\prime}\to\tau\to0$$ and $P\in \hat{C}$ such that the fiber of $p$ over $P$ is one dimensional we have that $$length(\tau_P)>0.$$
\end{enumerate}
Let us explain Condition 3 above. Denote by $0\to\mathcal{K}\to\mathcal{O}(-l)^{P(l)}\stackrel{j}{\rightarrow}\mathcal{F}\to0$ and $0\to\mathcal{K}\to\mathcal{O}(-l)^{P(l)}\stackrel{q}{\rightarrow}\mathcal{F}^{\prime}\to0$ the tautological sequences on  $\mathfrak{X}_l$,  $\mathfrak{X}_m$ respectively and $f:\mathcal{O}(-l)^{P(l)}\to\mathcal{O}(-m)^{P(m)}$ the tautological morphism. Then condition 3 translates in $q\circ f\circ j=0$. This is an open condition. Then Let $\mathfrak{A}_{l,m}=\mathfrak{X}_l\times_{\hat{\mathcal{C}}}\mathfrak{H}$ where the morphism $\mathfrak{X}_m\to\hat{\mathcal{C}}$ is the composition $\mathfrak{X}_m\to\mathcal{C}\stackrel{p}{\rightarrow}\hat{\mathcal{C}}$. Let $\mathfrak{B}$ be the locally closed locus in $\mathfrak{A}$ such that
\begin{enumerate}
\item the universal quotient $\mathcal{B}$ on $\mathfrak{X}_m$ is a locally free
\item the restriction of $\rho$ to $\mathcal{U}$  is an isomorphism
\item $p_*\mathcal{B}=\mathcal{F}$.
\end{enumerate}
We have that $G_{l,m}={\bf Gl}_l\times {\bf Gl}_m$ acts on $\mathfrak{A}$. We would now like to take $\cup_{l,m}\mathfrak{A}_{l,m}/G_{l,m}$ but in general such a quotient would only define a 2-stack. Let us further sketch how to describe $\mathfrak{P}$ as a quotient of a scheme by a group. The construction is rather standard. Take $\mathcal{H}^{\prime}$ as in \cite{mop} Section 6.1. More precisely $\mathcal{H}^{\prime}$ is a subscheme of $Hilb(\pp(V))\times\pp(V)^n$, where $Hilb(\pp(V))$ is the Hilbert scheme of of genus $g$ curves and degree $F=1 
-g + k(d + 1)(2g -2 + n) + kd $ for $k\geq5$ in the projective space $\pp(V)$ with $V\simeq\C^F$. Then $\mathcal{H}^{\prime}$ comes equipped with a universal curve $\mathcal{C}^{\prime}$, a contraction of rational tails $\mathcal{C}^{\prime}\to\hat{\mathcal{C}}^{\prime}$ and a ${\bf PGL}(V)$ action. If in the above construction we replace $\mathcal{C}$ by $\mathcal{C}^{\prime}$ and $\hat{\mathcal{C}}$ by $\hat{\mathcal{C}}^{\prime}$ we get a scheme $A_{l,m}=X_l\times_{\hat{\mathcal{C}}^{\prime}}H$ with $H$ a subscheme of $$Hom_{\hat{X}_l\times_{\hat{\mathcal{C}}^{\prime}} \hat{X}_m}(\mathcal{O}(-l)^{P(l)},\mathcal{O}(-m)^{P(m)}).$$ Then $\mathfrak{P}$ is the stack quotient $\cup_{l,m}[A_{l,m}/G_{l,m}\times {\bf PGL}(V)]$. This concludes the proof.
\end{proof}
\begin{remark} By the stability condition for stable maps implies that the restriction of $S^{\vee}$ to any rational tail has positive degree. This together with condition 4 in Construction \ref{funcp} shows that the restriction of $\hat{S}^{\vee}$ to any unstable component of $\hat{C}$ has positive degree. This means that $\omega(\sum p_i)\otimes \wedge^k S^{\epsilon}$ is ample for any $\epsilon>0$.
\end{remark}

\begin{lemma}There exists a morphism 
$$r:\mathfrak{Bun}^{bal}_{g,n}(k,d)\to\mathfrak{P}$$
\end{lemma}
\begin{proof}
This is essentially Proposition 7.2 in \cite{mihnea}. Let $$r(\mathcal{S}^{\vee})=(\mathcal{S}^{\vee},p_*\mathcal{S}^{\vee}\stackrel{\rho}{\rightarrow} p_*(\mathcal{S}(-a_iD_i))^{\vee})$$ where $\rho$ is the map induced by $\mathcal{S}^{\vee}\stackrel{D^{a_i}}{\rightarrow} \mathcal{S}^{\vee}(a_iD_i)$. The map is well defined since $p_*\mathcal{S}(-a_iD_i)$ is a vector bundle by Lemma 1.7 in \cite{mihnea}.
\end{proof}
\begin{construction} Let $\widetilde{\mathfrak{Bun}}_{g,n}(k,d)$ be the irreducible component of $\mathfrak{P}$ which contains the image of $r$.
\end{construction}
\begin{remark}Note that as $d$ becomes large $\mathfrak{P}$ might not be irreducible.
\end{remark}
\begin{lemma}\label{dimension}The stack $\widetilde{\mathfrak{Bun}}_{g,n}(k,d)$ has pure dimension equal to the dimension of $\mathfrak{Bun}_{g,n}(k,d)$.
\end{lemma}
\begin{proof} Let $\mathfrak{U}$ be the product $\widetilde{\mathfrak{Bun}}_{g,n}(k,d)\times_{\mathfrak{M}_{g,n}}\mathfrak{M}^{sm}_{g,n}$, where $\mathfrak{M}^{sm}_{g,n}$ is the locus in $\mathfrak{M}_{g,n}$ of smooth curves. Then $\mathfrak{U}$ is an open substack of $\widetilde{\mathfrak{Bun}}_{g,n}(k,d)$. As $\widetilde{\mathfrak{Bun}}_{g,n}(k,d)$ is irreducible it is enough to show that $\mathfrak{U}$ has pure dimension equal to the dimension of $\mathfrak{Bun}_{g,n}(k,d)$. For this we show that $\mathfrak{U}\simeq \mathfrak{Bun}_{g,n}(k,d)\times_{\mathfrak{M}_{g,n}}\mathfrak{M}^{sm}_{g,n}$. Objects of $\mathfrak{U}$ are pairs $(\mathcal{S}^{\vee},\mathcal{S}^{\vee}\stackrel{\rho}{\rightarrow}\mathcal{S}^{\vee})$ with $\rho$ an isomorphism and any such object is isomorphic to $(\mathcal{S}^{\vee},\mathcal{S}^{\vee}\stackrel{id}{\rightarrow}\mathcal{S}^{\vee})$ by composing $\rho$ with $\rho^{-1}$. It is clear that any isomorphism of $(\mathcal{S}^{\vee},\mathcal{S}^{\vee}\stackrel{id}{\rightarrow}\mathcal{S}^{\vee})$ induces an isomorphism of $\mathcal{S}$. Vice versa for any automorphism $\theta$ of $\mathcal{S}$ there exists a unique automorphism $\hat{\theta}=\theta$ of $\mathcal{S}$ such that the diagram commutes
\begin{equation*}\xymatrix{\mathcal{S}\ar[r]^{id}\ar[d]_{\theta}&\mathcal{S}\ar[d]^{\hat{\theta}}\\
\mathcal{S}\ar[r]^{id}&\mathcal{S}.}
\end{equation*}
This concludes the proof.
\end{proof}
\begin{remark}We have a diagram
\begin{equation*}
\xymatrix{&\widetilde{\mathfrak{Bun}}_{g,n}(k,d)\ar[ld]_{\pi_1}\ar[rd]^{\pi_2}\\
\mathfrak{Bun}_{g,n}(k,d)\ar@{-->}[rr]^r&&\mathfrak{Bun}^{rtf}_{g,n}(k,d)}
\end{equation*}
such that $r\circ\pi_1=\pi_2$.
\end{remark}
\begin{proof}
We have a morphism 
\begin{align*}\mathfrak{P}&\to \mathfrak{Bun}_{g,n}(k,d)\times_{\mathfrak{M}_{g,n}} \mathfrak{Bun}_{g,n}(k,d) \\
(p:\mathcal{C}\to\hat{\mathcal{C}},\mathcal{S},p_*\mathcal{S}^{\vee}\to \hat{\mathcal{S}}^{\vee})&\mapsto(\mathcal{S}, \hat{\mathcal{S}}).
\end{align*} We define $\pi_i$ to be the composition of the natural projection of the product to the $i$'th factor composed with the above map.
\end{proof}

\section{Stable map-quotients}\label{map-quot}
\subsection{Construction}
In this section we construct a proper algebraic stack which surjects to both moduli spaces of stable maps to $G(k,n)$ and stable quotients.

\begin{definition}\label{def} Let $\bar{P}$ be the category whose objects are $$(p:\mathcal{C}\to\hat{\mathcal{C}}, \mathcal{O}_{\mathcal{C}}^{\oplus r}\stackrel{s}{\rightarrow}\mathcal{S}^{\vee}, p_*(\mathcal{S}^{\vee})\stackrel{\rho}{\rightarrow} \hat{\mathcal{S}}^{\vee})$$ such that
\begin{enumerate} 
\item $p:\mathcal{C}\to\hat{\mathcal{C}}$ is the morphism of Lemma \ref{contraction} and $\mathcal{U}$ is as in Notation \ref{abuse}
\item $\mathcal{O}_{\mathcal{C}}^{\oplus r}\stackrel{s}{\rightarrow}\mathcal{S}^{\vee},$ is a stable map of rank $k$ and degree $d$.

\item $\hat{\mathcal{S}}$ is a rank $k$ degree $d$ locally free sheaf
\item the morphism of sheaves $\rho: p_*\mathcal{S}^{\vee}\to \hat{\mathcal{S}}^{\vee}$ is an isomorphism on $\mathcal{U}$
\item $\omega(\sum p_i)\otimes \wedge^k \hat{S}^{\epsilon}$ is ample for any $\epsilon>0$.
\item for each morphism $$p_*(S^{\vee})\stackrel{\rho}{\rightarrow} \hat{S}^{\vee}\to\tau\to0$$ and $P\in \hat{C}$ such that the fiber of $p$ over $P$ is one dimensional we have that $$length(\tau_P)>0.$$

\end{enumerate}
{\bf Isomorphisms:} An isomorphism
\begin{equation*}
\psi:(\mathcal{C}\stackrel{p}{\rightarrow}\hat{\mathcal{C}},\mathcal{O}_{\mathcal{C}}^{\oplus r}\stackrel{s}{\rightarrow}\mathcal{S}^{\vee},p_*(\mathcal{S}^{\vee})\stackrel{\rho}{\rightarrow}\hat{\mathcal{S}}^{\vee})\to(\mathcal{C}^{\prime} \stackrel{p^{\prime}}{\rightarrow}\hat{\mathcal{C}}^{\prime},\mathcal{O}_{\mathcal{C}^{\prime}}^{\oplus r}\stackrel{s^{\prime}}{\rightarrow}{\mathcal{S}^{\prime}}^{\vee},p^{\prime}_*({\mathcal{S}^{\prime}}^{\vee})\stackrel{\rho^{\prime}}{\rightarrow}\hat{\mathcal{S}^{\prime}}^{\vee})
\end{equation*}
is an automorphism of curves
\begin{align*}
\phi:\mathcal{C}&\to\mathcal{C}^{\prime}\\
\hat{\phi}:\hat{\mathcal{C}}&\to\hat{\mathcal{C}}^{\prime}
\end{align*}
together with isomorphims  and $\theta: \mathcal{S}^{\vee}\to \phi^*{\mathcal{S}^{\prime}}^{\vee}$
 and $\hat{\theta}: \hat{\mathcal{S}}^{\vee}\to \phi^*\hat{\mathcal{S}^{\prime}}^{\vee}$
 such that
\begin{enumerate}
\item $\phi(p_i)=p_i^{\prime}$ , $\forall i$
\item $p^{\prime}\circ\phi=\hat{\phi}\circ p$
\item the following diagram commutes
\begin{equation*}
\xymatrix{\mathcal{O}_{\mathcal{C}}^{\oplus r}\ar[d]\ar[r]^s&\mathcal{S}^{\vee}\ar[d]^{\theta}\\
\phi^*\mathcal{O}_{\mathcal{C}^{\prime}}^{\oplus r}\ar[r]^{\phi^*s^{\prime}}&\phi^*{\mathcal{S}^{\prime}}^{\vee}}
\end{equation*}
\item the following diagram commutes
\begin{equation*}
\xymatrix{p_*\mathcal{S}^{\vee}\ar[r]^{\rho}\ar[d]_{\hat{\theta}_p}&\hat{\mathcal{S}}^{\vee}\ar[d]^{\hat{\theta}}\\
\hat{\phi}^*p_*{\mathcal{S}^{\prime}}^{\vee}\ar[r]^{\hat{\phi}^*\rho^{\prime}}&\hat{\phi}^*\hat{\mathcal{S}^{\prime}}^{\vee}}
\end{equation*}
where $\hat{\theta}_p:p_*\mathcal{S}^{\vee}\to\hat{\phi}^*p_*{\mathcal{S}^{\prime}}^{\vee}$ is the isomorphism of sheaves induced by $\theta$ followed by the isomorphism $p_*\phi^*{\mathcal{S}^{\prime}}^{\vee}\simeq\hat{\phi}^*p_*{\mathcal{S}^{\prime}}^{\vee}$ from Lemma \ref{cbpopa}.
\end{enumerate}
We call a point in $\bar{P}$ a map-quotient or for short an m\&q.
\end{definition}
\begin{remark}\label{bounded}
Here we do not need to impose that the number of components of $C$ is bounded as for any stable map $(C, p_1,...,p_n,f)$ there exists an $N$ depending on $g$, $n$ and $d$ such that $C$ has at most $N$ components.
\end{remark}
\begin{remark}\label{equiv} Let us show that Condition 2 in Definition \ref{def} is equivalent to the one in \cite{mop}. Given a stable quotient $\hat{S}\to\mathcal{O}^{\oplus r}$ we get by dualizing a morphism $\mathcal{O}^{\oplus r}\to\hat{S}^{\vee}$ which is generically surjective in all fibers. Conversely, given a morphism $\mathcal{O}^{\oplus r}\to\hat{S}^{\vee}$ which is generically surjective in fibers we get an injective morphism $\hat{S}\to\mathcal{O}^{\oplus r}$ whose quotient is flat by Lemma \ref{geninj}.
\end{remark}

\begin{proposition}We have an isomorphism of stacks $$\bar{P}\simeq \overline{M}_{g,n}(\G(k,r),d)\times_{\mathfrak{Bun}_{g,n}(k,d)} \mathfrak{P}$$ where $N$ in the definition of $\mathfrak{P}$ is a number as in Remark \ref{bounded}. In particular, $\bar{P}$ is an Artin stack.
\end{proposition}
\begin{proof} It follows easily from definitions.
\end{proof}
\begin{lemma} We have a morphism of stacks $$R:\overline{M}_{g,n}(\G(k,r),d)^{bal}\to \bar{P}$$ over $\mathfrak{Bun}_{g,n}(k,d)$.
\end{lemma}
\begin{proof} Let $(\mathcal{O}^r\to\mathcal{S}^{\vee})\in \overline{M}_{g,n}(\G(k,r),d)^{bal}(B)$. We define $$R(\mathcal{O}^r\to\mathcal{S}^{\vee})=(\mathcal{O}^r\to\mathcal{S}^{\vee},p_*\mathcal{S}^{\vee}\stackrel{\rho}{\rightarrow} (p_*\mathcal{S}(-a_iD_i))^{\vee}).$$ Let us show that the map is well defined. Let us consider the composition $0\to \mathcal{S}(-a_iD_i)\to \mathcal{S}\to\mathcal{O}^{\oplus r}$ and let $\mathcal{Q}^{\prime}$ be the quotient. By Lemma \ref{geninj} we have that $\mathcal{Q}^{\prime}$ is flat over $B$. We have that $R^1p_*\mathcal{S}(-a_iD_i)=0$ which implies that the sequence
\begin{equation*}
0\to p_*\mathcal{S}(-aD)\to \mathcal{O}^{\oplus r}\to p_*\mathcal{Q}^{\prime}\to0
\end{equation*}
is exact. This shows that $p_*\mathcal{S}(-a_iD_i)$ is a vector bundle.
\end{proof}
\begin{example}\label{popa} Let us prove that in general $R$ does not extend. We repeat the argument in \cite{mihnea}, Theorem 7.4. Let us consider a 1-dimensional constant family of stable genus zero maps $f^1:\mathcal{C}^1_B\to\G(k,r)$ with a constant section $s^1:B\to\mathcal{C}^1$, $s^1(b)=P$, for any point $b\in B$. Let $f^0:\mathcal{C}_B^0\to \G(k,r)$ be a family of genus zero maps, such that on the general fiber $C_b\to\G(k,r)$ the pull-back of the tautological subbundle is isomorphic to $\mathcal{O}(-1)\oplus...\oplus\mathcal{O}(-1)$ and on the special fiber $C^0_0\to\G(k,r)$ the pull-back of the tautological subbundle is isomorphic to $\mathcal{O}\oplus\mathcal{O}(-2)\oplus\mathcal{O}(-1)...\oplus\mathcal{O}(-1)$. Let us consider a section $s^0:B\to\mathcal{C}^0$ such that $f^1(s^1(B))=f^0(s^0(B))$. By identifying $f^1(s^1(B))$ with $f^0(s^0(B))$ we obtain a family of stable maps $f:\mathcal{C}\to\G(k,n)$. Let $f_0:C_0\to\G(k,r)$ be the special fiber and $p$ the gluing point of $C_0^1$ with $C_0^0$. Let  $\{s_1 , . . . , s_k \}$ be a local basis for $S|_{C_0^1}$ at $p$. If $x$ is a local coordinate around $p$, then $\{xs_1,..., xs_k\}$ is a basis of $\hat{S}$.
\\ Let us now consider a second family of stable maps. As $\overline{M}_{0,n}(\G(k,r),d)$ is irreducible we can find a family of stable maps $B^{\prime}$ whose special fiber is $f_0:C_0\to\G(k,r)$ and general fiber with smooth domain. By proposition 7.3 in \cite{mihnea} the limiting stable quotient over $B^{\prime}$ is $\hat{S}$ with $\hat{S}$ around $p$ generated by $\{s_1,x^2s_2,xs_3..., xs_k\}$ As the two quotients are different $R$ does not extend.
\end{example}
\begin{proposition} We have a morphisms of stacks $$i:\bar{P}\to\overline{M}_{g,n}(\G(k,r),d)\times_{\mathfrak{M}^{rtf}_{g,n}}\overline{Q}_{g,n}(\G(k,r),d)$$ which is an immersion of DM stacks. In particular, we have morphisms of stacks
\begin{align*}
&c_1:\bar{P}\to\overline{M}_{g,n}(\G(k,r),d)\\
&c_2:\bar{P}\to\overline{Q}_{g,n}(\G(k,r),d).
\end{align*}
\end{proposition}
\begin{proof} The exact sequence $0\to \mathcal{Q}^{\vee}\to\mathcal{O}^r\to \mathcal{S}^{\vee}\to0$ induces a an exact sequence $$0\to p_*\mathcal{Q}^{\vee}\to\mathcal{O}^r\to p_*(\mathcal{S}^{\vee})\to R^1p_*\mathcal{Q}^{\vee}.$$ Let us show that $R^1p_*\mathcal{Q}^{\vee}$ is supported on the complement of $U$ in $\hat{\mathcal{C}}$. We have that $R^1p_*\mathcal{Q}^{\vee}|_U=R^1p_*(\mathcal{Q}^{\vee}|_U)$ and since $p$ is the identity on $U$ we obtain that $R^1p_*\mathcal{Q}^{\vee}|_U=0$. Composing the generically surjective morphism $\mathcal{O}^r\to p_*(\mathcal{S}^{\vee})$ with $p_*(\mathcal{S}^{\vee})\to\hat{\mathcal{S}^{\vee}}$ we obtain an element in $\overline{Q}_{g,n}(\G(k,r),d)$.
\\Let us now show that given a stable map $m=(C, \mathcal{O}^r\stackrel{s}{\rightarrow} \mathcal{S}^{\vee}\to 0)$ and a stable quotient $q=(\hat{C}, \mathcal{O}^r\stackrel{\hat{s}}{\rightarrow} \hat{\mathcal{S}}^{\vee})$ there exists at most one $mq=(p:\mathcal{C}\to\hat{\mathcal{C}}, \mathcal{O}_{\mathcal{C}}^{\oplus r}\stackrel{s}{\rightarrow}\mathcal{S}^{\vee}, p_*(\mathcal{S}^{\vee})\stackrel{\rho}{\rightarrow} \hat{\mathcal{S}}^{\vee})$ such that $i(mq)=(m,q)$. As the map $p_*s:p_*\mathcal{O}^r\stackrel{}{\rightarrow}p_* \mathcal{S}^{\vee}$ induced by $s$ is surjective away from torsion and zero on the torsion we obtain that a map $\rho: p_* \mathcal{S}^{\vee}\to\hat{\mathcal{S}}^{\vee}$ such that $\rho\circ p_*s=\hat{s}$ must be unique.
\end{proof}
\begin{corollary} $\bar{P}$ is separated.
\end{corollary}
\begin{construction} Let $$\overline{MQ}_{g,n}(\G(k,r),d)=\overline{M}_{g,n}(\G(k,r),d)\times_{\mathfrak{Bun}_{g,n}(k,d)} \widetilde{\mathfrak{Bun}}_{g,n}(k,d).$$ Then by construction $\overline{MQ}_{g,n}(\G(k,r),d)$ is a substack of $\bar{P}$ and it contains the image of $\bar{R}$.
\end{construction}
\subsection{Properties}
Let us first show that $\bar{P}$ is proper.
\paragraph{Properness} 
\begin{lemma}\label{tor} Let $\mathcal{C}\to B$ be a flat family of curves. Let us assume that $B$, $\mathcal{C}$ are smooth and let $D$ be a $-1$ curve on $\mathcal{C}$. Then we have that $$Tor_1(\mathcal{O}_D(-a),\mathcal{O}_D)\simeq\mathcal{O}_D(-a+1)$$ for any $a>0$. 
\end{lemma}
\begin{proof} We have an exact sequence
\begin{equation*}
0\to \mathcal{O}(-D)\to \mathcal{O}\to\mathcal{O}_D\to0.
\end{equation*}
Let $a>0$. Tensoring it with $\mathcal{O}(aD)$ we obtain the following free resolution of $\mathcal{O}_D(-a)$
\begin{equation*}
0\to \mathcal{O}((a-1)D)\to \mathcal{O}(aD)\to\mathcal{O}_D(-a)\to0.
\end{equation*}
Tensoring the resolution with $\mathcal{O}_D$ we obtain a sequence
 \begin{equation*}
\mathcal{O}_D((-a+1)D)\to \mathcal{O}_D(-aD)\to\mathcal{O}_D(-a)\to0.
\end{equation*}
This shows that $Tor_1(\mathcal{O}_D(-a),\mathcal{O}_D)=\mathcal{O}_D(-a+1)$.
\end{proof}
\begin{lemma} \label{nattrans}Let $\mathcal{C}\to B$ a family of curves as before, $D$ be a $(-1)$-curve and $$0\to\mathcal{S}\to\mathcal{O}_{\mathcal{C}}^{\oplus n}\stackrel{q}{\rightarrow} Q\to0$$ a flat quotient on $\mathcal{C}$. Let $q:\mathcal{C}\to\mathcal{C}^{\prime}$ be the contraction of $D$. Then there exists a morphism of sheaves $\mathcal{S}^{\prime}\to\mathcal{S}$ on $\mathcal{C}^{\prime}$ which is an isomorphism on $\mathcal{C}\backslash D$.
\end{lemma}
\begin{proof} Let $S|_D\simeq \mathcal{O}(-a_1)\oplus...\oplus \mathcal{O}(-a_k)$ with $a_i\geq 0$ and let us assume that $a_1\leq...\leq a_k$. Let us prove that there exists a vector bundle $\mathcal{T}$ on $\mathcal{C}$ such that the following conditions are fulfilled 
\begin{enumerate}
\item we have a morphism $\mathcal{T}\to\mathcal{S}$ which is an isomorphism on $\mathcal{C}\backslash D$
\item $\mathcal{T}|_D\simeq \mathcal{O}^{\oplus k}$ 
\end{enumerate}
Let $\mathcal{T}_1$ be the kernel of the natural projection $\mathcal{S}\to \mathcal{O}_D(-a_k)$ so that we have an exact sequence
\begin{equation*}
0\to \mathcal{T}_1\to\mathcal{S}\to \mathcal{O}_D(-a_k)\to 0.
\end{equation*} Restricting the above sequence to $D$ we have an exact sequence
\begin{equation}
0\to Tor_1(\mathcal{O}_D(-a_k),\mathcal{O}_D)\to \mathcal{T}_1\to\mathcal{S}\to \mathcal{O}_D(-a_k)\to 0.
\end{equation}
By Lemma \ref{tor} we have that $Tor_1(\mathcal{O}_D(-a_k),\mathcal{O}_D)\simeq \mathcal{O}_D(-a+1)$ which shows that we have an exact sequence 
\begin{equation}\label{induction}0\to\mathcal{O}_D(-a_k+1)\to\mathcal{T}_1\to \mathcal{O}(-a_1)\oplus...\oplus \mathcal{O}(-a_{k-1})\to0.
\end{equation} Let us assume that $\mathcal{T}_1|_D\simeq \mathcal{O}(-b_1)\oplus...\oplus \mathcal{O}(-b_k)$ with $b_1\leq...\leq b_k$. By sequence (\ref{induction}) we have that $b_i\geq0$ and $\sum_{i=1}^kb_i<\sum_{i=1}^ka_i$. Repeating the above procedure for $\mathcal{T}_1$ instead of $\mathcal{S}$ we obtain $\mathcal{T}$ as in the above claim.
\\Let $\mathcal{S}^{\prime}=q_*\mathcal{T}$. We have that $\mathcal{S}^{\prime}$ is a vector bundle over $\mathcal{C}^{\prime}$. By construction $q_*\mathcal{T}\to q_*\mathcal{S}$ is an isomorphism on $U$.
\end{proof}
\begin{remark} Let $$0\to\mathcal{S}\to\mathcal{O}_{\mathcal{C}}^{\oplus n}\stackrel{q}{\rightarrow} Q\to0$$ be a flat quotient on $\mathcal{C}$. If $\mathcal{S}^{\prime}$ is the vector bundle constructed in the above lemma, then we obtain a flat family of quotients over $\mathcal{C}^{\prime}$. Indeed we have that the composition $\mathcal{S}^{\prime}\to q_*\mathcal{S}\to \mathcal{O}^{\oplus r}$ in injective on all fibers of $\mathcal{C}\to B$. By Lemma \ref{geninj} we have that $\mathcal{S}^{\prime}\to\mathcal{O}^{\oplus r}\to \mathcal{Q}^{\prime}$ is a flat family of quotients over $\mathcal{C}^{\prime}$.
\end{remark}

\textit{Proof of Properness.} Let $(p:\mathcal{C}^*\to\hat{\mathcal{C}}^*,\mathcal{S}^*\stackrel{j^*}{\rightarrow} \mathcal{O}_{\mathcal{C}^*}^{\oplus r}, p_*({\mathcal{S}^*}^{\vee})\stackrel{\rho^*}{\rightarrow} (\hat{\mathcal{S}}^*)^{\vee})$ be a family of m\&q's over $\Delta^*$. By normalizing and possibly restricting $\Delta$ we may assume that $\hat{\mathcal{C}}^*$  is smooth. By the properness of the moduli space of stable maps we can extend $0\to \mathcal{S}^*\stackrel{j^*}{\rightarrow} \mathcal{O}_{\mathcal{C}^*}^{\oplus r}\to Q^*\to 0$ to $$0\to\mathcal{S}\stackrel{j}{\rightarrow} \mathcal{O}_{\mathcal{C}}^{\oplus r}\to Q\to0$$ with $Q$ flat over $\Delta$. Let $\mathcal{C}_0$ be (reducible) 2-dimensional component of $\mathcal{C}$ contracted by $p$. Let $\mathcal{C}^{\prime}$ be the closure of the complement of $\mathcal{C}_0$ in $\mathcal{C}$, $p^{\prime}:\mathcal{C}\to\mathcal{C}^{\prime}$ be the projection to $\mathcal{C}^{\prime}$ and $\mathcal{S}^{\prime}$ the restriction of $\mathcal{S}$ to $\mathcal{C}^{\prime}$. By the semistable reduction theorem and by possibly blowing up the nodes of the central fiber we may assume that $\mathcal{C}^{\prime}$ is smooth. Without loss of generality we may assume that $\mathcal{C}^{\prime}$ does not have unstable fibers over $\Delta^{\star}$. By the properness of the moduli space of stable maps we have a family $0\to\mathcal{S}\to\mathcal{O}_{\mathcal{C}}^r\to\mathcal{Q}\to0$ over $\Delta$. Restricting to $\mathcal{C}^{\prime}$ we obtain a family of stable maps $$0\to\mathcal{S}^{\prime}\to\mathcal{O}_{\mathcal{C}^{\prime}}^r\to\mathcal{Q}^{\prime}\to0.$$ We have the following exact sequence
\begin{equation}\label{torsion}
0\to T\stackrel{i}{\rightarrow} p_*^{\prime}\mathcal{S}^{\vee}\stackrel{q}{\rightarrow}{\mathcal{S}^{\prime}}^{\vee}\to 0
\end{equation}
where $T$ is the torsion subsheaf of $p_*^{\prime}\mathcal{S}^{\vee}$. By our assumption $p^{\prime}=p$ over $\Delta^*$. This gives that over $\Delta^*$ we have a morphism $$p_*^{\prime}\mathcal{S}^{\vee}\stackrel{\rho}{\rightarrow}\hat{\mathcal{S}}^{\vee}.$$ As $\rho\circ i(T)=0$, the universal property of quotients gives that $\rho$ factors (uniquely) through ${\mathcal{S}^{\prime}}^{\vee}$. This gives a morphism of locally free sheaves 
\begin{equation}\label{dualprime}
0\to{\mathcal{S}^{\prime}}^{\vee}\stackrel{\rho^{\prime}}{\rightarrow}\hat{\mathcal{S}}^{\vee}
\end{equation}
over $\Delta^*$ which is an isomorphism on $U^*$. Dualizing (\ref{dualprime}) and applying Lemma \ref{geninj} we have that $\hat{\mathcal{S}}\to\mathcal{S}^{\prime}$ is a flat family of quotients over $\Delta^*$ and by the properness of $Quot$ scheme we have 
\begin{equation}
0\to\hat{\mathcal{S}}^{\prime}\to\mathcal{S}^{\prime}
\end{equation} 
over $\mathcal{C}^{\prime}$. Dualizing again we have 
\begin{equation}\label{almost}
0\to{\mathcal{S}^{\prime}}^{\vee}\stackrel{\rho^{\prime}}{\rightarrow}\hat{\mathcal{S}^{\prime}}^{\vee}\to\tau\to0
\end{equation} is a flat family of quotients which extends (\ref{dualprime}). As $\mathcal{C}^{\prime}$ is smooth and ${\hat{\mathcal{S}^{\prime}}}^{\vee}$ reflexive we obtain that ${\hat{\mathcal{S}^{\prime}}}^{\vee}$ is is locally free. By (\ref{torsion}) and (\ref{almost}) we have  
\begin{equation}\label{gata}
p_*^{\prime}\mathcal{S}^{\vee}\to \hat{\mathcal{S}^{\prime}}^{\vee}\to\tau\to0
\end{equation} over $\mathcal{C}^{\prime}$ which satisfies the conditions 1-4 from Construction \ref{def}. Let us prove it also satisfies condition 5. Let $\tau_1,...,\tau_m$ be the irreducible components of $\tau$ and $\mathcal{C}_{0,1},...,\mathcal{C}_{0,m}$ be the irreducible components of $\mathcal{C}_0$. As $length(\tau_i)>0$ over $\Delta^*$ for all $i$ we obtain that $length(\tau_i)>0$ over $\Delta$. This implies that condition 5 in Definition \ref{def} is true for the points $P\in \mathcal{C}^{\prime}\cap \mathcal{C}_0$. 
\\We now contract the remaining unstable curves in the central fiber. Let $p^{\prime\prime}:\mathcal{C}^{\prime}\to \mathcal{C}^{\prime\prime}$ be the contraction of \emph{one} unstable rational tail and let $P$ the attachment point with the rest of the curve. Let us show that we can find a vector bundle $\mathcal{S}^{\prime\prime}$ and a morphism of sheaves 
\begin{equation}\label{onecontr}
p^{\prime\prime}_*(\hat{\mathcal{S}^{\prime}}^{\vee})\stackrel{\rho^{\prime\prime}}{\rightarrow}{\mathcal{S}^{\prime\prime}}^{\vee}
\end{equation}
on $\mathcal{C}^{\prime}$ which satisfies the conditions 1-4 from Definition \ref{def} and the restriction of $\mathcal{S}^{\prime\prime}$ to the fiber of $\mathcal{C}^{\prime\prime}\to B$ which contains $P$ is strictly negative. By Lemma \ref{nattrans} we have a bundle $\mathcal{T}$ with a morphism $\mathcal{T}\to \hat{\mathcal{S}}^{\prime}$. By construction we have that the restriction of $\mathcal{T}$ to the fiber of $\mathcal{C}^{\prime\prime}\to B$ which contains $P$ is strictly negative. Dualizing and pushing forward we obtain a morphism $$p^{\prime\prime}_*(\hat{\mathcal{S}^{\prime}}^{\vee})\to p^{\prime\prime}_*(\mathcal{T}^{\vee})$$ which is an isomorphism on $U$. By construction $\mathcal{T}|_D$ is trivial which implies that $p^{\prime\prime}_*(\mathcal{T}^{\vee})$ is a vector bundle. Moreover, the restriction of $p^{\prime\prime}_*(\mathcal{T}^{\vee})$ to the fiber of $\mathcal{C}^{\prime\prime}\to B$ which contains $P$ is strictly positive. This shows that we can take $\mathcal{S}^{\prime\prime}=p^{\prime\prime}_*(\mathcal{T}^{\vee})$. Repeating this procedure for all unstable rational tails we may assume that (\ref{onecontr}) holds for $p^{\prime\prime}:\mathcal{C}^{\prime}\to \hat{\mathcal{C}}$ the contraction of \emph{all} unstable rational tails. This means that $p=p^{\prime\prime}\circ p^{\prime}$. Take $\hat{\mathcal{S}}=\mathcal{S}^{\prime\prime}$ and by (\ref{gata}) we get a morphism $$\rho:p_*({\mathcal{S}}^{\vee})\to\hat{\mathcal{S}}^{\vee}$$ on $\hat{\mathcal{C}}$ as in construction (\ref{def}). In the end we contract -2 curves on which $\mathcal{S}$ and $\hat{\mathcal{S}}$ are trivial. 

\begin{corollary} The stack $\overline{MQ}_{g,n}(\G(k,r),d)$ is a proper DM stack.
\end{corollary}
\begin{proof} By definition $\overline{MQ}_{g,n}(\G(k,r),d)$ is a closed substack of $\bar{P}$. As $\bar{P}$ is proper, we obtain that $\overline{MQ}_{g,n}(\G(k,r),d)$ is proper.
\end{proof}
\begin{proposition}\label{summary} Let us consider $$\mathfrak{Q}_{g,n}(\G(k,r),d):=\overline{Q}_{g,n}(\G(k,r),d)\times_{\mathfrak{Bun}_{g,n}(k,d)}\widetilde{\mathfrak{Bun}}_{g,n}(k,d).$$ Then we have the following commutative diagram

\begin{equation*}
\xymatrix{
&{\overline{MQ}_{g,n}(\G(k,r),d)}\ar[d]_q\ar[ld]_{c_1}\ar[rd]^{c_2}\\
\overline{M}_{g,n}(\G(k,r),d)\ar[dd]_{\nu_M}&\mathfrak{Q}_{g,n}(\G(k,r),d)\ar[d]\ar[r]^s&\overline{Q}_{g,n}(\G(k,r),d)\ar[d]^{\nu_Q}\\
&\widetilde{\mathfrak{Bun}}_{g,n}(k,d)\ar[ld]_{\pi_1}\ar[r]^{\pi_2}&\mathfrak{Bun}_{g,n}(k,d)\\
\mathfrak{Bun}_{g,n}(k,d).}
\end{equation*}
\end{proposition}
\begin{proof}
The definitions imply that the following diagram is commutative
\begin{equation*}
\xymatrix{{\overline{M}_{g,n}(\G(k,r),d)}\ar[r]\ar[d]&\overline{Q}_{g,n}(\G(k,r),d)\ar[d]\\
\widetilde{\mathfrak{Bun}}_{g,n}(k,d)\ar[r]&\mathfrak{Bun}_{g,n}(k,d).}
\end{equation*}
By the universal property of cartesian products we obtain a map $$q:\overline{M}_{g,n}(\G(k,r),d)\to \mathfrak{Q}_{g,n}(\G(k,r),d).$$

\end{proof}
\begin{remark} Let $C$ be nodal curve and $C_0$ be its rational tail. If $$0\to S\to\mathcal{O}^r_C$$ is a stable map from $C$ to $\G(k,r)$, then the map $q$ forgets the map $$S_{|C_0}\to\mathcal{O}^r_{C_0}.$$ If $\overline{M}_{g,n}(\G(k,r),d)$ has components whose points are maps with rational tails, then $q$ is not an isomorphism. 
\end{remark}

\subsection{Obstruction theories}\label{oth}
Let us shortly recall a few basic facts about obstruction theories of moduli spaces of stable maps to $\G(k,r)$ and stable quotients. Let $$\epsilon_M: \overline{M}_{g,n}(\G(k,r),d)\to \mathfrak{M}_{g,n}$$ be the morphism that forgets the map (and does not stabilize the pointed curve) and $$\pi_M: \mathcal{C}_M\to \overline{M}_{g,n}(\G(k,r),d)$$ the universal curve over $\overline{M}_{g,n}(\G(k,r),d)$ and let $ev$ denote the evaluation map $ev:\overline{M}_{g,n+1}(\G(k,r),d)\to \G(k,r)$ (see \cite{b}). Let $$0\to\mathcal{S}_M\to\mathcal{O}_{\mathcal{C}_M}^n\to\mathcal{Q}_M\to0$$ be the universal sequence on $\mathcal{C}_M$. We have that $ev^*T_{\G(k,r)}\simeq \mathcal{Q}_M\otimes\mathcal{S}_M^{\vee}$. This shows that $$E_{\overline{M}_{g,n}(\G(k,r),d)/\mathfrak{M}}^{\bullet}:=\mathcal{R}^{\bullet}{\pi_M}_*\mathcal{Q}_M\otimes\mathcal{S}_M^{\vee}$$ is a dual obstruction theory for the morphism $p$ (see \cite{bf}). We call $$[\overline{M}_{g,n}(\G(k,r),d)]^{\vv}:={\epsilon_M}_{\mathfrak{E}_{\overline{M}_{g,n}(\G(k,r),d)/\mathfrak{M}}}^![\mathfrak{M}_{g,n}]$$ the virtual class of $\overline{M}_{g,n}(\G(k,r),d)$. Here ${\epsilon_M}_{\mathfrak{E}_{\overline{M}_{g,n}(\G(k,r),d)/\mathfrak{M}}}^!$ is the virtual pull-back in \cite{eu}.

As the moduli space of stable maps, the moduli space of stable quotients $\bar{Q}_{g,m}(\G(r,n),d)$ has a morphism $\epsilon_Q:\bar{Q}_{g,m}(\G(r,n),d)\to\mathfrak{M}_{g,m}$to the Artin stack of nodal curves. Let $\pi_Q:\hat{\mathcal{C}}_Q\to\bar{Q}_{g,m}(\G(r,n),d)$ be the universal curve over $\bar{Q}_{g,m}(\G(r,n),d)$ and let $$0\to\hat{\mathcal{S}}_Q\to\mathcal{O}_{\hat{\mathcal{C}}_Q}^n\to\hat{\mathcal{Q}}_Q\to0$$ be the universal sequence on $\hat{\mathcal{C}}_Q$. Then the complex $$E^{\bullet}_{\overline{Q}_{g,n}(\G(k,r),d)/\mathfrak{M}}=R{\pi_Q}_*RHom(\hat{\mathcal{S}},\hat{\mathcal{Q}}_M)$$ is a dual obstruction theory relative to $\epsilon_Q$. We call $$[\overline{Q}_{g,n}(\G(k,r),d)]^{\vv}:={\epsilon_Q}_{\mathfrak{E}_{\overline{Q}_{g,n}(\G(k,r),d)/\mathfrak{M}}}^![\mathfrak{M}_{g,n}]$$ the virtual class of $\overline{Q}_{g,n}(\G(k,r),d)$.

\paragraph{Obstruction theories relative moduli spaces of bundles.} In the following we define obstruction theories relative to $\mathfrak{Bun}_{g,n}(k,d)$. The map $$\nu_M:\overline{M}_{g,n}(\G(k,r),d)\to \mathfrak{Bun}_{g,n}(k,d)$$ induces a morphism between cotangent complexes and thus we obtain a distinguished triangle
\begin{equation*}
\nu_M^*L_{\mathfrak{Bun}_{g,n}(k,d)}\to L_{\overline{M}_{g,n}(\G(k,r),d)}\to  L_{\overline{M}_{g,n}(\G(k,r),d)/\mathfrak{Bun}_{g,n}(k,d)}.
\end{equation*}
Tensoring the tautological sequence on the universal curve over $\overline{M}_{g,n}(\G(k,r),d)$ with $\mathcal{S}_M^{\vee}$ we obtain an exact sequence
\begin{equation*}
0\to \mathcal{S}_M\otimes\mathcal{S}_M^{\vee}\to(\mathcal{S}_M^{\vee})^{\oplus r}\to \mathcal{Q}_M\otimes \mathcal{S}_M^{\vee}\to 0
\end{equation*}
which induces a distinguished triangle
\begin{equation*}
R^{\bullet}{\pi_M}_*(\mathcal{S}_M^{\vee})^{\oplus r}\to R^{\bullet}{\pi_M}_*\mathcal{Q}_M\otimes \mathcal{S}_M^{\vee}\to R^{\bullet}{\pi_M}_*\mathcal{S}_M\otimes\mathcal{S}_M^{\vee}[1].
\end{equation*}
By the Cohomology and base change theorem we obtain that $\nu_M^*T_{\mathfrak{Bun}_{g,n}(k,d)}=R^{\bullet}{\pi_M}_*\mathcal{S}_M\otimes\mathcal{S}_M^{\vee}[1]$. This shows that we have the following commutative diagram
\begin{equation*}
\xymatrix{T_{\overline{M}_{g,n}(\G(k,r),d)/\mathfrak{Bun}_{g,n}(k,d)}\ar[r]\ar[d]&T_{\overline{M}_{g,n}(\G(k,r),d)}\ar[r]\ar[d]&\nu_M^*T_{\mathfrak{Bun}_{g,n}(k,d)}\ar@{=}[d]\\
R^{\bullet}{\pi_M}_*(\mathcal{S}_M^{\vee})^{\oplus r}\ar[r] &R^{\bullet}{\pi_M}_*\mathcal{Q}_M\otimes \mathcal{S}_M^{\vee}\ar[r]& R^{\bullet}{\pi_M}_*\mathcal{S}_M\otimes\mathcal{S}_M^{\vee}[1]}
\end{equation*}
and therefore $R^{\bullet}{\pi_M}_*(\mathcal{S}_M^{\vee})^{\oplus r}$ is a dual relative obstruction theory for $\nu_M$.
\\In a completely analogous manner we obtain that $R^{\bullet}{\pi_M}_*(\hat{\mathcal{S}}_Q^{\vee})^{\oplus r}$ is a dual relative obstruction theory for $\nu_Q$.
\begin{construction}\label{relobs} From the cartesian diagram in Proposition \ref{summary} we obtain that $E^{\bullet}_{\overline{MQ}_{g,n}(\G(k,r),d)/\widetilde{\mathfrak{Bun}}_{g,n}(k,d)}=c_1^*R^{\bullet}{\pi_M}_*(\mathcal{S}_M^{\vee})^{\oplus r}$ is a dual perfect obstruction for the map $$\overline{MQ}_{g,n}(\G(k,r),d)\to\widetilde{\mathfrak{Bun}}_{g,n}(k,d).$$ By Lemma \ref{dimension} $\widetilde{\mathfrak{Bun}}_{g,n}(k,d)$ has pure dimension, which means that $c_1^*R^{\bullet}{\pi_M}_*(\mathcal{S}_M^{\vee})^{\oplus r}$ gives rise to a virtual class $$[\overline{MQ}_{g,n}(\G(k,r),d)]^{\vv}=(\nu_{MQ})^!_{\mathfrak{E}_{\overline{MQ}_{g,n}(\G(k,r),d)/\widetilde{\mathfrak{Bun}}_{g,n}(k,d)}}[\widetilde{\mathfrak{Bun}}_{g,n}(k,d)].$$
\end{construction}
\section{Comparison of virtual fundamental classes}\label{compvf}
\begin{proposition}\label{morfism} The tautological morphism $$\rho:p_*(\mathcal{S}_{MQ}^{\vee})\to\hat{\mathcal{S}}^{\vee}_{MQ}$$ on $\overline{MQ}_{g,n}(\G(k,r),d)$ induces a morphism $$c_1^*R^{\bullet}{\pi_M}_*(\mathcal{S}_M^{\vee})^{\oplus r} \to c_2^*R^{\bullet}{\pi_Q}_*(\hat{\mathcal{S}}_Q^{\vee})^{\oplus r}.$$
\end{proposition}
\begin{proof}
By the construction of $\overline{MQ}_{g,n}(\G(k,r),d)$ we see that we have a commutative diagram with the right-down square cartesian
\begin{equation}\label{diagobs}
\xymatrix{\mathcal{C}_{MQ}\ar[dr]^p\ar[drr]^{c_2^{\prime\prime}}\ar[ddr]_{\pi_{MQ}}\\
&\hat{\mathcal{C}}_{MQ}\ar[r]^{c_2^{\prime}}\ar[d]_{t}&\hat{\mathcal{C}}_Q\ar[d]^{\pi_Q}\\
&\overline{MQ}_{g,n}(\G(k,r),d)\ar[r]^{c_2}&\overline{Q}_{g,n}(\G(k,r),d).}
\end{equation}
By cohomology and base change in the diagram
\begin{equation*}
\xymatrix{\mathcal{C}_{MQ}\ar[r]\ar[d]_{\pi_{MQ}}&\mathcal{C}_M\ar[d]^{\pi_{M}}\\
\overline{MQ}_{g,n}(\G(k,r),d)\ar[r]^{c_1}&\overline{M}_{g,n}(\G(k,r),d)}
\end{equation*} 
we have that $c_1^*R^{\bullet}{\pi_M}_*\mathcal{S}_M^{\vee}\simeq R^{\bullet}{\pi_{MQ}}_*c_1^*\mathcal{S}_M^{\vee}$ and by construction we have that $c_1^*\mathcal{S}_M^{\vee}\simeq \mathcal{S}_{MQ}^{\vee}$. Combining the two relations we obtain a canonical isomorphism 
\begin{equation}\label{r1}c_1^*R^{\bullet}{\pi_M}_*\mathcal{S}_M^{\vee}\simeq R^{\bullet}{\pi_{MQ}}_*\mathcal{S}_{MQ}^{\vee}.
\end{equation}
From the commutativity of diagram (\ref{diagobs}) we have that 
\begin{equation}\label{r2}R^{\bullet}{\pi_{MQ}}_*\mathcal{S}_{MQ}^{\vee}\simeq R^{\bullet}(p\circ t)_*\mathcal{S}_{MQ}^{\vee}
\end{equation} and by the construction of $\overline{MQ}_{g,n}(\G(k,r),d)$ we have $\hat{\mathcal{S}}_{MQ}^{\vee}\simeq {c_2^{\prime}}^*\hat{\mathcal{S}}_Q^{\vee} $.
Using now cohomology and base change in diagram (\ref{diagobs}) we obtain that
\begin{equation}\label{r3}c_2^*R^{\bullet}{\pi_Q}_*\hat{\mathcal{S}}_Q^{\vee}\simeq R^{\bullet}t_*{c_2^{\prime}}^*\hat{\mathcal{S}}_Q^{\vee}.
\end{equation}
By (\ref{r1}), (\ref{r2}) and (\ref{r3}) we see that $\rho:p_*(\mathcal{S}_{MQ}^{\vee})\to\hat{\mathcal{S}}_{MQ}$ induces a morphism $$c_1^*R^{\bullet}{\pi_M}_*(\mathcal{S}_M^{\vee})^{\oplus r} \to c_2^*R^{\bullet}{\pi_Q}_*(\hat{\mathcal{S}}_Q^{\vee})^{\oplus r}.$$
\end{proof}

\begin{lemma}\label{perffibers} Let $F$ be the cone of the morphism $$c_1^*R^{\bullet}{\pi_M}_*(\mathcal{S}_M^{\vee})^{\oplus r} \to c_2^*R^{\bullet}{\pi_Q}_*(\hat{\mathcal{S}}_Q^{\vee})^{\oplus r}.$$ Then, $F$ is a perfect complex.
\end{lemma}
\begin{proof}
Let us consider $$f:(C,x_1,...,x_{m})\to\G$$ a stable map, $p:C\to\widehat{C}$ the morphism contracting the rational tails and let $x_1,...,x_p$ be the gluing points of the rational tails $C^0_i$ with the rest of the curve.
Then we need to show that the morphism
\begin{equation*}
H^1(C, f^*S^{\vee})\to H^1(\hat{C}, \hat{S}^{\vee}(\sum d_ix_i))
\end{equation*}
is surjective. Since $$H^1(C, f^*S^{\vee})\simeq H^1(\hat{C}, p_*f^*S^{\vee})$$ we need to show that 
\begin{equation*}
H^1(\widehat{C}, p_*f^*S^{\vee})\to H^1(\hat{C},p_* f^*S^{\vee}(\sum d_iC^0_i))
\end{equation*}
is surjective. As the quotient of the morphism $q_*f^*S^{\vee}\to p_*f^*S^{\vee}(\sum d_iC^0_i)$ is supported on points, it has no higher cohomology. This shows that the above morphism is surjective.
\end{proof}

\begin{theorem}\label{final} Let $\gamma_1,...,\gamma_n\in A^*(\G(k,r)).$ Then we have that $$ev_1^*\gamma_1...ev_n^*\gamma_n\cdot[\overline{M}_{g,n}(\G(k,r),d)]^{\vv}=ev_1^*\gamma_1...ev_n^*\gamma_n\cdot[\overline{Q}_{g,n}(\G(k,r),d)]^{\vv}.$$ 
\end{theorem}
\begin{proof} From Costello's push forward formula \cite{co} applied to the cartesian diagram 
\begin{equation*}
\xymatrix{\overline{MQ}_{g,n}(\G(k,r),d)\ar[r]^-{c_1}\ar[d]&\overline{M}_{g,n}(\G(k,r),d)\ar[d]\\
\widetilde{\mathfrak{Bun}}_{g,n}(k,d)\ar[r]&\mathfrak{Bun}_{g,n}(k,d)}
\end{equation*} 
we obtain that
\begin{equation}\label{leftd}
{c_1}_*[\overline{MQ}_{g,n}(\G(k,r),d)]^{\vv}=[\overline{M}_{g,n}(\G(k,r),d)]^{\vv},
\end{equation}
where $[\overline{MQ}_{g,n}(\G(k,r),d)]^{\vv}$ is the class from Construction \ref{relobs}. Let us now analyze the commutative (not cartesian) diagram
\begin{equation*}
\xymatrix{\overline{MQ}_{g,n}(\G(k,r),d)\ar[r]^-{c_2}\ar[d]&\overline{Q}_{g,n}(\G(k,r),d)\ar[d]\\
\widetilde{\mathfrak{Bun}}_{g,n}(k,d)\ar[r]&\mathfrak{Bun}_{g,n}(k,d).}
\end{equation*} 
Let $\mathfrak{B}\to \widetilde{\mathfrak{Bun}}_{g,n}(k,d)$ be a desingularization of $\widetilde{\mathfrak{Bun}}_{g,n}(k,d)$. Then we can construct a cartesian diagram
\begin{equation*}
\xymatrix{M\ar[r]^-s\ar[d]&\overline{MQ}_{g,n}(\G(k,r),d)\ar[d]\\
\mathfrak{B}\ar[r]&\widetilde{\mathfrak{Bun}}_{g,n}(k,d).}
\end{equation*} 
The relative obstruction theory of $\nu_M$ induces a virtual class on $M$ and by Costello's push-forward formula we have that $$s_*[M]^{\vv}=[\overline{MQ}_{g,n}(\G(k,r),d)]^{\vv}.$$ With this, we have shown that by replacing $\widetilde{\mathfrak{Bun}}_{g,n}(k,d)$ with $\mathfrak{B}$ we may assume that $\widetilde{\mathfrak{Bun}}_{g,n}(k,d)$ is smooth. As the moduli space of stable quotients is connected (\cite{kp, toda}) Proposition \ref{morfism}  and Lemma \ref{perffibers} show that we are under the hypothesis of Proposition 3.14 in \cite{eu2} which implies that
\begin{equation}\label{rightd}
{c_2}_*[\overline{MQ}_{g,n}(\G(k,r),d)]^{\vv}=[\overline{Q}_{g,n}(\G(k,r),d)]^{\vv}.
\end{equation}
The conclusion follows from equations (\ref{leftd}), (\ref{rightd}) and the projection formula. 
\end{proof}

\vspace{0.3cm}
Imperial College London, 180 Queen's Gate, SW7 2AZ London
\\{\it E-mail address:} \texttt{c.manolache$@$imperial.ac.uk}


\begin{thebibliography}{99}
\bibitem{am} A Bayer, Y Manin, Stability conditions, wall-crossing and weighted Gromov-Witten invariants. Mosc. Math. J. {\bf 9} (2009), no. 1, 3Ð32
\bibitem{b} K. Behrend, Gromov-Witten invariants in algebraic geometry. Invent. Math. \textbf{127} (1997), no. 3, 601--617.
\bibitem{bf} K. Behrend, B. Fantechi, The intrinsic normal cone. Invent. Math. \textbf{127} (1997), no.1, 45--88.
\bibitem{fga} B. Fantechi, L. G\"ottsche, L. Illusie, Luc, S. Kleiman, N. Nitsure, A.; Vistoli, Fundamental algebraic geometry. Grothendieck's FGA explained. Mathematical Surveys and Monographs, 123. American Mathematical Society, Providence, RI, 2005.
\bibitem{cf1} I. Ciocan-Fontanine, B. Kim, Moduli stacks of stable toric quasimaps. Adv. Math. {\bf 225} (2010), no. 6, 3022Ð3051.
\bibitem{cf2} I. Ciocan-Fontanine, B. Kim, D. Maulik,  Stable quasimaps to GIT quotients, 2011, arXiv: 1106.3724.
\bibitem{chs1}  I Coskun, J Harris, J Starr, The ample cone of the Kontsevich moduli space. Canad. J. Math. {\bf 61} (2009), no. 1, 109Ð123
\bibitem{chs2} I Coskun, J Harris, J Starr, The effective cone of the Kontsevich moduli space. Canad. Math. Bull. {\bf 51} (2008), no. 4, 519Ð534.
\bibitem{chs3} I Coskun, J Starr, Divisors on the space of maps to Grassmannians. Int. Math. Res. Not. 2006, Art. ID 35273, 25 pp.
\bibitem{co} K. Costello, Higher genus Gromov-Witten invariants as genus zero invariants of symmetric products. Ann. of Math. (2) {\bf 164} (2006), no. 2, 561Ð601.
\bibitem{gi}  A. Givental, A mirror theorem for toric complete intersections, in ÒTopological 
field theory, primitive forms and related topics (Kyoto, 1996)Ó, Progr. Math., 
160, Birkhuser Boston, Boston, MA, 1998, 141Ð175. 
\bibitem{lg} J. Li, G. Tian, Virtual moduli cycles and Gromov-Witten invariants of general symplectic manifolds. Topics in symplectic $4$-manifolds (Irvine, CA, 1996), 47--83
\bibitem{kp} B. Kim, R. Pandharipande, The connectedness of the moduli space of maps to homogeneous spaces,  Symplectic geometry and mirror symmetry (Seoul, 2000),  187--201, World Sci. Publ., River Edge, NJ, 2001.




\bibitem{ml} M. Lieblich, Remarks on the stack of coherent algebras. Int. Math. Res. Not. 2006, Art. ID 75273, 12 pp,

\bibitem{eu} C. Manolache, Virtual pull-backs, Journal of Algebraic Geometry, {\bf 21} (2012) 201-245.
\bibitem{eu2} C. Manolache, Virtual push-forwards,  Geometry \& Topology {\bf 16} (2012), 2003-2036.
\bibitem{mop} A. Marian, D. Oprea, R. Pandharipande, The moduli space of stable quotients, Geometry\& Topology, {\bf 15} (2011), no. 3, 1651Ð1706
\bibitem{mm} A. Musta\c{t}\u{a}, Andrei, A. Musta\c{t}\u{a}, Intermediate moduli spaces of stable maps. Invent. Math. {\bf 167} (2007), no. 1, 47Ð90.
\bibitem{mihnea}  M. Popa, M. Roth, Stable maps and Quot schemes. Invent. Math. {\bf 152} (2003), no. 3, 625--663.
\bibitem{toda} Y. Toda, Moduli spacesof stable quotients and the wall-crossing phenomena, Compos. Math. {\bf 147} (2011), no. 5, 1479Ð1518.

\end{thebibliography}
\end{document}